\documentclass[reqno]{amsart}

\usepackage[sorting=none, sorting=nyt, maxbibnames=99, backend=biber]{biblatex}

\renewbibmacro{in:}{}
\bibliography{references.bib}
\usepackage{amssymb}
\usepackage{calrsfs}
\usepackage{graphics,graphicx, mathrsfs}
\usepackage{enumerate}
\usepackage{enumitem}
\usepackage{url}
\usepackage{xcolor}
\definecolor{vio}{rgb}{0.54, 0.17, 0.89}
\usepackage[ruled, lined, linesnumbered, longend]{algorithm2e}
\usepackage{hyperref}
\usepackage{titlesec}

\newtheorem{theorem}{Theorem}[section]
\newtheorem{lemma}[theorem]{Lemma}
\newtheorem{question}[theorem]{Question}

\newtheorem{corollary}[theorem]{Corollary}
\numberwithin{equation}{section}

\theoremstyle{remark}
\newtheorem*{remark}{Remark}

\titleformat{\section}
  {\normalfont\large\bfseries\centering}{\thesection}{1em}{}

\titleformat{\subsection}
  {\normalfont\bfseries}{\thesubsection}{1em}{}

\DeclareMathOperator{\li}{\mathrm{li}}

\def\reals{\hbox{\rm I\kern-.18em R}}
\def\complexes{\hbox{\rm C\kern-.43em
\vrule depth 0ex height 1.4ex width .05em\kern.41em}}
\def\field{\hbox{\rm I\kern-.18em F}} 

\let\svthefootnote\thefootnote
\newcommand\freefootnote[1]{%
  \let\thefootnote\relax%
  \footnotetext{#1}%
  \let\thefootnote\svthefootnote%
}

\newenvironment{section*}[2][A]{
  \section*{#2}
  \renewcommand\thesection{#1}
  \setcounter{theorem}{0}}{}

\allowdisplaybreaks

\begin{document}

\title[The optimality of the error term in the prime number theorem]{Zero-density estimates and the optimality of the error term in the prime number theorem}

\author{Daniel R. Johnston}
\address{School of Science, UNSW Canberra, Australia}
\email{daniel.johnston@unsw.edu.au}
\date\today
\keywords{}

\begin{abstract}
    We demonstrate the impact of a generic zero-free region and zero-density estimate on the error term in the prime number theorem. Consequently, we are able to improve upon previous work of Pintz and provide an essentially optimal error term for some choices of the zero-free region. As an example, we show that if there are no zeros $\rho=\beta+it$ of $\zeta(s)$ with
    \begin{equation*}
        1-\beta<\frac{1}{c(\log t)^{2/3}(\log\log t)^{1/3}}=:\eta(t),
    \end{equation*}
    then
    \begin{equation*}
        \frac{|\psi(x)-x|}{x}\ll\exp(-\omega(x))\frac{(\log x)^9}{(\log\log x)^3},
    \end{equation*}
    where $\psi(x)$ is the Chebyshev prime-counting function, and
    \begin{equation*}
        \omega(x)=\min_{t\geq 3}\{\eta(t)\log x+\log t\}.
    \end{equation*}
    This refines the best known error term for the prime number theorem, previously given by 
    \begin{equation*}
        \frac{|\psi(x)-x|}{x}\ll_{\varepsilon}\exp(-(1-\varepsilon)\omega(x))
    \end{equation*} 
    for any $\varepsilon>0$.
\end{abstract}

\maketitle

\freefootnote{\textit{Affiliation}: School of Science, The University of New South Wales Canberra, Australia.}
\freefootnote{\textit{Key phrases}: zero-free region, zero-density estimates, prime number theorem}
\freefootnote{\textit{2020 Mathematics Subject Classification}: 11M26, 11N05 (Primary) 11N56 (Secondary)}

\section{Introduction}
A central problem in analytic number theory is to obtain good estimates for the prime counting functions
\begin{equation*}
    \pi(x)=\sum_{p\leq x}1,\quad \theta(x)=\sum_{p\leq x}\log p\quad\text{and}\quad \psi(x)=\sum_{\substack{p^k\leq x\\ k\geq 1}}\log p,
\end{equation*}
where each index $p$ is prime. The prime number theorem asserts that
\begin{equation*}
    \pi(x)\sim\li(x):=\int_2^x\frac{\mathrm{d}t}{\log t},\quad\theta(x)\sim x\quad\text{and}\quad\psi(x)\sim x.
\end{equation*}
Therefore, one typically seeks strong bounds on the (scaled) error terms
\begin{equation}\label{deltadef}
    \Delta_{1}(x)=\frac{|\pi(x)-\li(x)|}{x/\log x},\quad \Delta_{2}(x)=\frac{|\theta(x)-x|}{x}\quad\text{and}\quad\Delta_{3}(x)=\frac{|\psi(x)-x|}{x}
\end{equation}
noting that $\li(x)\sim x/\log x$. Assuming the Riemann hypothesis, Schoenfeld \cite{schoenfeld1976sharper} proved the bound
\begin{equation*}
    \Delta_{i}(x)<\frac{\log^2x}{8\pi\sqrt{x}},\quad x\geq 2657
\end{equation*}
for all $i\in\{1,2,3\}$. Without the Riemann hypothesis however, one is left with much weaker bounds for $\Delta_{i}(x)$. In this case, to optimise bounds on $\Delta_{i}(x)$ one needs to extract as much information as possible about the zeros of the Riemann-zeta function $\zeta(s)$. This is typically done through \emph{zero-free regions} and \emph{zero-density estimates}.

A zero-free region is equivalent to the existence of a continuous decreasing function $\eta(t)$ with $0<\eta(t)\leq 1/2$ such that $\zeta(s)$ has no zeros $\rho=\beta+it$ with
\begin{equation*}
    \beta>1-\eta(|t|).
\end{equation*}
Here, the general goal is to find the largest such function $\eta(t)$. Typical forms of $\eta(t)$ include
\begin{equation}\label{zfregions}
    \eta_1(t)=\frac{1}{R\log t}\quad\text{and}\quad\eta_2(t)=\frac{1}{c(\log t)^{2/3}(\log\log t)^{1/3}}    
\end{equation}
for sufficiently large $t$ and constants $R>0$ and $c>0$. Functions of the form $\eta_1(t)$ and $\eta_2(t)$ are respectively called \emph{classical} and \emph{Vinogradov-Korobov} zero-free regions, with the latter named after the work of Vinogradov \cite{vinogradov1958new} and Korobov \cite{korobov1958estimates}. Certainly, zero-free regions of the form $\eta_2(t)$ are asymptotically larger than those of the form $\eta_1(t)$. However in practice, the value of $R$ is much lower than $c$ in \eqref{zfregions} so that $\eta_1(t)$ is superior to $\eta_2(t)$ for small to modest values of $t$ (see \cite{mossinghoff2024explicit}). Moreover, the machinery required to prove a classical zero-free region $\eta_1(t)$ is simpler, meaning an analogous region is also likelier to hold for $L$-functions more general than $\zeta(s)$ (see e.g.\ \cite[Theorem 1.9]{weiss1983least}).

A zero-density estimate is then an upper bound on the quantity
\begin{equation}\label{zerodendef}
    N(\sigma,t):=\#\{\rho=\beta+i\tau:\zeta(\rho)=0,\:\sigma<\beta<1,\: 0<\tau<t\},
\end{equation}
with this bound preferably as small as possible. Often, one can obtain zero-density estimates of the form
\begin{equation}
    N(\sigma,t)\ll t^{A(1-\sigma)^B}(\log t)^C,
\end{equation}
for constants $A>0$, $B\geq 1$, $C\geq 0$, uniformly over a range $\sigma\in[\sigma_0,1]$ with $\sigma_0\in(1/2,1)$. For some examples, see \cite[Chapter 11]{ivic2003riemann}.

In his 1932 book, Ingham \cite{ingham1932distribution} demonstrated a link between zero-free regions and the error term in the prime number theorem. For such a result, one defines
\begin{equation}
    \omega(x):=\min_{t\geq 1}\{\eta(t)\log x+\log t\},
\end{equation}
for a choice of zero-free region $\eta(t)$. Then, assuming some natural conditions on $\eta(t)$, Ingham proved that \cite[Theorem 22]{ingham1932distribution}
\begin{equation}\label{inghameq}
    \Delta_i(x)\ll_{\varepsilon} \exp\left(-\left(\frac{1}{2}-\varepsilon\right)\omega(x)\right)
\end{equation}
for any $\varepsilon>0$ and $i\in\{1,2,3\}$. In 1980, this was improved upon by Pintz \cite[Theorem 1]{pintz1980remainder} who utilised a simple zero-density estimate of Carlson \cite{carlson1920nullstellen} to prove
\begin{equation}\label{Pintzeq}
    \Delta_i(x)\ll_{\varepsilon} \exp(-(1-\varepsilon)\omega(x)).
\end{equation}
In many applications, Pintz's refinement \eqref{Pintzeq} yields far better results than Ingham's \eqref{inghameq}. Consequently, recent explicit estimates for the error terms $\Delta_i(x)$, have used Pintz's method or a variant thereof \cite{platt2021error,broadbent2021sharper,johnston2023some,fiori2023sharpera,fiori2023sharperb}. 

Notably, Pintz was also able to prove \cite[Theorem 2]{pintz1980remainder} that if there exists infinitely many zeros $\rho=\beta+it$ of $\zeta(s)$ such that 
\begin{equation*}
    \beta\geq 1-\eta(|t|),
\end{equation*}
then,
\begin{equation*}
    \Delta_i(x)=\Omega_{\pm}\bigg(\exp(-(1+\varepsilon)\omega(x))\bigg).
\end{equation*}
This shows that Pintz's bound \eqref{Pintzeq} has very little room for improvement. However, we can still ask the following natural question.
\begin{question}\label{mainquest}
    If one has a zero-free region no larger than $\eta(t)$, is it true that
    \begin{equation*}
        \Delta_i(x)\ll\exp(-\omega(x))?
    \end{equation*}
\end{question}
In particular, is it possible to remove the $\varepsilon$ from \eqref{Pintzeq} and thereby, from a zero-free region, give an essentially optimal error term in the prime number theorem? In this paper we will explore this question and in some cases, provide an affirmative answer. Our general approach will be to refine Pintz's method whilst also employing a more general zero-density estimate.

\section{Statement of results}
Our main theorem, which we prove in Section \ref{mainsect}, is as follows. The result demonstrates how a (mostly) arbitrary zero-free region and zero-density estimate affects the error term in the prime number theorem.
\begin{theorem}\label{mainthm}
    Let $\eta(t)$ be a decreasing function for $t\geq 0$ with a continuous derivative $\eta'(t)$ such that $0<\eta(t)\leq 1/2$ and $\zeta(s)$ has no zeros $\rho=\beta+it$ with
    \begin{equation}\label{zerofreeassum}
        \beta>1-\eta(|t|).
    \end{equation}
    Also assume that $N(\sigma,t)$, defined in \eqref{zerodendef}, satisfies a bound
    \begin{equation}\label{zerodenassum}
        N(\sigma,t)\ll t^{A(1-\sigma)^B}(\log t)^C
    \end{equation}
    for some $A>0$, $B\geq 1$, $C\geq 0$, uniformly over a range $\sigma\in[\sigma_0,1]$ with $\sigma_0\in(1/2,1)$. In addition, suppose that for sufficiently large $x$, the function
    \begin{equation}\label{fdef}
        f_x(t):=\eta(t)\log x+\log t
    \end{equation}
    is minimised at a unique value $t=t_0$. Then if 
    \begin{equation}\label{omegadef}
        \omega(x):=f_x(t_0)
    \end{equation}
    satisfies 
    \begin{equation}\label{omegacons}
        \omega(x)\geq 2\log\log x\quad\text{and}\quad \omega(x)=o(\log x),    
    \end{equation}
    one has
    \begin{equation}\label{mainbound}
        \Delta_{i}(x)\ll\exp(-\omega(x))\exp\left(2A\omega(x)\left(\frac{\omega(x)}{\log x}\right)^B\right)\omega(x)^C
    \end{equation}
    for all $i\in\{1,2,3\}$ as $x\to\infty$. Here, the implied constant depends on the choice of zero-free region $\eta(t)$, and zero-density estimate \eqref{zerodenassum}. 
\end{theorem}
Care is taken in the proof of Theorem \ref{mainthm} so as to not assume any zero-free region or zero-density estimate beyond \eqref{zerofreeassum} and \eqref{zerodenassum}. We remark that the conditions on $\eta(t)$ (and consequently $\omega(x)$) are rather easy to satisfy. For example, they are satisfied by any zero-free region of the form
\begin{equation*}
    \eta(t)=\frac{c_1}{(\log t)^{c_2}(\log\log t)^{c_3}},
\end{equation*}
where $c_1>0$, $c_2>0$ and $c_3\in\mathbb{R}$ are constants, and $t$ is sufficiently large. We also note that the value of $2$ in \eqref{mainbound} could be lowered with more work. However, this would have no impact on our subsequent results \eqref{logfreeasym} and \eqref{vkasym}. 

We now give some examples of Theorem \ref{mainthm}. To begin with, we use a log-free zero-density estimate of the form
\begin{equation}\label{logfreeeq}
    N(\sigma,t)\ll t^{A(1-\sigma)}.
\end{equation}
In particular, the work of Jutila \cite[Theorem 1]{jutila1977linnik} allows one to take\footnote{Or in fact $A=2+\varepsilon$ for any $\varepsilon>0$.} $A=5/2$ in \eqref{logfreeeq}, uniformly for $\sigma\in[0.8,1]$. This gives the following corollary.

\begin{corollary}\label{logfreecor}
    Let $\eta(t)$ and $\omega(x)$ satisfy the conditions of Theorem \ref{mainthm}. Then
    \begin{equation*}
        \Delta_i(x)\ll\exp(-\omega(x))\exp\left(5\frac{(\omega(x))^2}{\log x}\right)
    \end{equation*}
    for all $i\in\{1,2,3\}$ as $x\to\infty$. In particular, if $\omega(x)=O(\sqrt{\log x})$ then
    \begin{equation}\label{logfreeasym}
        \Delta_i(x)\ll\exp(-\omega(x)).
    \end{equation}
\end{corollary}
Notably, if
\begin{equation*}
    \eta(t)=\frac{1}{R\log t},\qquad R>0
\end{equation*}
is a classical zero-free region, then standard calculus arguments give 
\begin{equation*}
    \omega(x)=\frac{2}{\sqrt{R}}\sqrt{\log x}=O(\sqrt{\log x}).
\end{equation*}
Thus, Corollary \ref{logfreecor} implies that Question \ref{mainquest} can be answered in the affirmative if $\eta(t)$ is a classical zero-free region (or weaker). For a Vinogradov-Korobov zero-free region, we instead consider a zero-density estimate of the form
\begin{equation}\label{vkzeroden}
    N(\sigma,t)\ll t^{A(1-\sigma)^{3/2}}(\log t)^C,
\end{equation}
whereby Ford \cite[p.~2]{ford2002vinogradov} gives $A=58.05$ and $C=15$ for $\sigma\in[9/10,1]$. Substituting these values into Theorem \ref{mainthm} yields the following result.
\begin{corollary}\label{vkcor}
    Let $\eta(t)$ and $\omega(x)$ satisfy the conditions of Theorem \ref{mainthm}. Then
    \begin{equation*}
        \Delta_i(x)\ll\exp(-\omega(x))\exp\left(117\frac{(\omega(x))^{5/2}}{(\log x)^{3/2}}\right)\omega(x)^{15}
    \end{equation*}
    for all $i\in\{1,2,3\}$ as $x\to\infty$. In particular, if
    \begin{equation}\label{vkzfeq}
        \eta(t)=\frac{1}{c(\log t)^{2/3}(\log\log t)^{1/3}},\qquad c>0,
    \end{equation}
    then 
    \begin{equation*}
        \omega(x)=\min_{t\geq 3}\{\eta(t)\log x+\log t\}\sim \left(\frac{5^6}{2^2\cdot 3^4\cdot c^3}\right)^{1/5}\frac{(\log x)^{3/5}}{(\log\log x)^{1/5}}
    \end{equation*}
    and 
    \begin{equation}\label{vkasym}
        \Delta_i(x)\ll\exp(-\omega(x))\frac{(\log x)^9}{(\log\log x)^3}.
    \end{equation}
\end{corollary}
To the best of the author's knowledge, Corollary \ref{vkcor} gives the sharpest known (unconditional) error term in the prime number theorem. Here, the lowest known value of $c$ in \eqref{vkzfeq} is $c=53.989$ \cite[Theorem 1.2]{bellotti2024explicit}. We also see that if the zero-density estimate \eqref{vkzeroden} could be made log-free, then one would attain an affirmative answer to Question \ref{mainquest} for Vinogradov-Korobov zero-free regions\footnote{Since the time of writing, Bellotti \cite{bellotti2025new} has released a preprint with a new zero-density estimate which is strong enough to give a log-free version of \eqref{vkasym}}. 

Finally, we remark that in both Corollary \ref{logfreecor} and Corollary \ref{vkcor}, the value of $A$ in the zero-density estimate does not affect the asymptotic form of \eqref{logfreeasym} and \eqref{vkasym}. This is in stark contrast to the problem of counting primes in short intervals, whereby the value of $A$ plays a central role (see \cite{starichkova2025note}).

\section{Preliminary lemmas}
We now provide some lemmas related to sums over the non-trivial zeros of $\zeta(s)$. First, we give a truncated version of the explicit Riemann--von Mangoldt formula.

\begin{lemma}[{See e.g.\ \cite[Theorem 12.5]{montgomery2006multiplicative}}]\label{RVMlem}
    For all $2\leq T\leq x$,
    \begin{equation}\label{rvmeq}
        \psi(x)=x-\sum_{|\Im(\rho)|\leq T} \frac{x^{\rho}}{\rho}+O\left(\frac{x(\log x)^2}{T}\right),
    \end{equation}
    as $x\to\infty$, where the sum is over the non-trivial zeros $\rho$ of $\zeta(s)$.
\end{lemma}
\begin{remark}
    Although sharper versions of the error term in \eqref{rvmeq} exist (e.g.\ \cite{Goldston_83, cully2023error,johnston2024error}), Lemma \ref{RVMlem} is sufficient for our purposes, and its standard proof does not assume any zero-free region or zero-density estimate.
\end{remark}

In order to use Lemma \ref{RVMlem}, we will require estimates for sums over $1/|\Im(\rho)|$.
\begin{lemma}[{See \cite[Theorem 25b]{ingham1932distribution}}]\label{Implem1}
    One has
    \begin{equation*}
        \sum_{1<|\Im(\rho)|\leq T}\frac{1}{|\Im(\rho)|}=O\left((\log T)^2\right).
    \end{equation*}
\end{lemma}

\begin{lemma}\label{Implem2}
    For some $A>0$, $B\geq 1$ and $C\geq 0$, suppose one has a zero-density estimate of the form \eqref{zerodenassum} uniformly over the range $\sigma\in[\sigma_0,1]$. If
    \begin{equation}\label{s1def}
        \sigma_1=\max\left\{\frac{A}{A+0.5},\sigma_0\right\},    
    \end{equation}
    then
    \begin{equation}\label{s1impest}
         \sum_{\substack{1<|\Im(\rho)|\leq T\\\sigma_1<\Re(s)<1}}\frac{1}{|\Im(\rho)|}= O(1)
    \end{equation}
    as $T\to\infty$.
\end{lemma}
\begin{proof}
    Since $\sigma_1\in (1/2,1)$ is fixed,
    \begin{equation}\label{ns1est}
        N(\sigma_1,t)\ll t^{A(1-\sigma_1)^B}(\log t)^C\ll t^{(A+1)(1-\sigma_1)}. 
    \end{equation}
    Now, by the vertical symmetry of the zeros of $\zeta(s)$, and integration by parts
    \begin{equation*}
        \sum_{\substack{1<|\Im(\rho)|\leq T\\\sigma_1<\Re(s)<1}}\frac{1}{|\Im(\rho)|}=2\int_1^T\frac{\mathrm{d}N(\sigma_1,t)}{t}=2\frac{N(\sigma_1,T)}{T}+2\int_1^T\frac{N(\sigma_1,t)}{t^2}\mathrm{d}t.
    \end{equation*}
    Using \eqref{ns1est} this is then bounded by
    \begin{equation*}
        2t^{A-(A+1)\sigma_1}+2\left[\frac{t^{A-(A+1)\sigma_1}}{A-(A+1)\sigma_1}\right]_1^T.
    \end{equation*}
    Here, $A-(A+1)\sigma_1\in(-1,0)$ by the definition \eqref{s1def} of $\sigma_1$. Hence \eqref{s1impest} follows. 
\end{proof}

\section{Proof of Theorem \ref{mainthm}}\label{mainsect}
We now prove Theorem \ref{mainthm}. Note that it suffices to prove the result for $\Delta_3(x)$ since the bound for $\Delta_2(x)$ will follow upon noting that \cite[Theorem 413]{hardy2008introduction}
\begin{equation*}
    \psi(x)=\theta(x)+O\left(x^{1/2}(\log x)^2\right)    
\end{equation*}
and the bound for $\Delta_1(x)$ can be obtained from that of $\Delta_2(x)$ by partial summation.
Our approach is structurally similar to the proofs of \cite[Theorem 1]{platt2021error} and \cite[Theorem 1.2]{fiori2023sharpera}, whereby $\Delta_3(x)$ is bounded in the case of a classical zero-free region.

\begin{proof}[Proof of Theorem \ref{mainthm}]
   Let $x$ be sufficiently large and set 
    \begin{equation}\label{Tdef}
        T=\exp(2\omega(x)).
    \end{equation}
    By the conditions \eqref{omegacons} on $\omega(x)$, we have $T\geq (\log x)^4$ and $T=o(x)$. Consequently, Lemma \ref{RVMlem} and the triangle inequality gives
    \begin{equation*}
        |\psi(x)-x|\leq \sum_{|\Im(\rho)|\leq T}\left|\frac{x^{\rho}}{\rho}\right|+O\left(\frac{x(\log x)^2}{T}\right)
    \end{equation*}
    so that
    \begin{equation}\label{rvmeq2}
        \Delta_3(x)\leq \sum_{|\Im(\rho)|\leq T}\frac{x^{\Re(\rho)-1}}{|\Im(\rho)|}+O\left(\frac{(\log x)^2}{T}\right).
    \end{equation}
    The bound $\omega(x)\geq 2\log\log x$ then implies that the error term in \eqref{rvmeq2} satisfies
    \begin{equation*}
        \frac{(\log x)^2}{T}= \exp(-\omega(x))\cdot\frac{(\log x)^2}{\exp(\omega(x))}\ll \exp(-\omega(x)).
    \end{equation*}
    Hence, it suffices to bound the sum in \eqref{rvmeq2}. To do so, we split the sum three times depending on the real part of $\rho$. We also only consider $1<|\Im(\rho)|\leq T$ as $\zeta(s)$ has no zeros with $0<|\Im(\rho)|\leq 1$ (see e.g.\ \cite[Chapter 6]{edwards1974riemanns}). In particular, we write
    \begin{align*}
        \sum_{|\Im(\rho)|\leq T}\frac{x^{\Re(\rho)-1}}{|\Im(\rho)|}&=\left(\sum_{\substack{1<|\Im(\rho)|\leq T\\ \Re(\rho)\leq \sigma_1}} +\sum_{\substack{1<|\Im(\rho)|\leq T\\ \sigma_1<\Re(\rho)\leq \sigma_2}}+\sum_{\substack{1<|\Im(\rho)|\leq T\\ \sigma_2<\Re(\rho)<1}}\right)\frac{x^{\Re(\rho)-1}}{|\Im(\rho)|}\\
        &=s_1(x)+s_2(x)+s_3(x),\quad\text{say},
    \end{align*}
    with
    \begin{equation}\label{s1s2def}
        \sigma_1=\max\left\{\frac{A}{A+0.5},\sigma_0\right\}\quad\text{and}\quad\sigma_2=1-\frac{\omega(x)}{\log x}.
    \end{equation}
    Now, before bounding each $s_i(x)$, we note that since $\omega(x)=o(\log(x))$,
    \begin{equation*}
        x^{-\varepsilon}\ll\exp(-\omega(x))
    \end{equation*}
    for any $\varepsilon>0$. Thus, by Lemma \ref{Implem1},
    \begin{equation*}
        s_1(x)\ll x^{\sigma_1-1}(\log T)^2\ll x^{\sigma_1-1}(\log x)^2\ll\exp(-\omega(x)).
    \end{equation*}
    Next, by Lemma \ref{Implem2} we have
    \begin{equation*}
        s_2(x)\ll x^{\sigma_2-1}=x^{-\omega(x)/\log x}=\exp(-\omega(x)).
    \end{equation*}
    So, to finish the proof it suffices to show that $s_3(x)$ satisfies the bound in \eqref{mainbound}. For this, we use the zero-free region $\eta(t)$ to deduce that
    \begin{equation*}
        s_3(x)\leq\sum_{\substack{1<|\Im(\rho)|\leq T\\ \sigma_2<\Re(\rho)<1}}\frac{x^{-\eta(|\Im(\rho)|)}}{|\Im(\rho)|}.
    \end{equation*}
    Consequently, by the symmetry of the zeros of $\zeta(s)$, and integration by parts,
    \begin{align}
        s_3(x)&\leq 2\int_1^T\frac{x^{-\eta(t)}}{t}\mathrm{d}N(\sigma_2,t)\notag\\
        &=2\frac{N(\sigma_2,T)x^{-\eta(T)}}{T}-2\int_1^T\frac{\mathrm{d}}{\mathrm{d}t}\left(\frac{x^{-\eta(t)}}{t}\right)N(\sigma_2,t)\mathrm{d}t\label{s2firsteq}.
    \end{align}
    Now,
    \begin{equation*}
        \frac{x^{-\eta(t)}}{t}=\exp\left(-f_x(t)\right)
    \end{equation*}
    where $f_x(t)$ is as defined in \eqref{fdef}. Hence, 
    \begin{equation*}
        \frac{\mathrm{d}}{\mathrm{d}t}\left(\frac{x^{-\eta(t)}}{t}\right)\leq 0\quad \text{if and only if $t\geq t_0$},
    \end{equation*}
    and 
    \begin{equation*}
        \frac{x^{\eta(t_0)}}{t_0}=\exp(-\omega(x))
    \end{equation*}
    by the definition \eqref{omegadef} of $\omega(x)$ and $t_0$. Therefore, by the fundamental theorem of calculus, \eqref{s2firsteq} is bounded above by
    \begin{align*}
        2\frac{N(\sigma_2,T)x^{-\eta(T)}}{T}+2N(\sigma_2,T)\int_{t_0}^T\frac{\mathrm{d}}{\mathrm{d}t}\left(-\frac{x^{-\eta(t)}}{t}\right)\mathrm{d}t&=2N(\sigma_2,T)\frac{x^{-\eta(t_0)}}{t_0}\\
        &=2N(\sigma_2,T)\exp(-\omega(x)).
    \end{align*}
    We now use the bound \eqref{zerodenassum} for $N(\sigma_2,T)$ to give
    \begin{equation*}
        s_3(x)\ll N(\sigma_2,T)\exp(-\omega(x))\ll\exp(-\omega(x))\exp\left(2A\omega(x)\left(\frac{\omega(x)}{\log x}\right)^B\right)\omega(x)^C
    \end{equation*}
    as desired.
\end{proof}

\section*{Acknowledgements}
I would like to thank Chiara Bellotti, Bryce Kerr, Valeriia Starichkova and Tim Trudgian for their helpful discussions on this paper. This research was supported by an Australian Government Research Training Program (RTP) Scholarship and an Australian Mathematical Society Lift-off Fellowship.

\newpage

\printbibliography

\end{document}